%% file: main.tex
\documentclass[11pt]{article}

\input{preamble}

\begin{document}

\title{The statistical threshold for planted matchings and spanning trees} 
\author{Louigi Addario-Berry\thanks{Department of Mathematics and Statistics, McGill University, Montr\'eal, Canada}
\and
Omer Angel\thanks{Department of Mathematics,
University of British Columbia, Vancouver, Canada}
\and
Gábor Lugosi\thanks{Department of Economics and Business, 
Pompeu  Fabra University, Barcelona, Spain;
ICREA, Pg. Lluís Companys 23, 08010 Barcelona, Spain;
Barcelona Graduate School of Economics}
\and
Miklós Z. Rácz\thanks{Department of Statistics and Data Science and Department of Computer Science, Northwestern University}
\and
Tselil Schramm\thanks{Department of Statistics, Stanford University}}

\maketitle

\begin{abstract} 
In this paper, we study the problem of detecting the presence of a planted perfect matching or spanning tree in an Erd\H{o}s--Rényi random graph. More precisely, we study the hypothesis testing problem where the statistician observes a graph on $n$ vertices. Under the null hypothesis, the graph is a realization of an Erd\H{o}s--Rényi random graph $G(n,q)$, while under the alternative hypothesis, the graph is the union of an 
Erd\H{o}s--Rényi random graph and a random perfect matching (or random spanning tree).
In order to avoid trivial detection by counting edges, we adjust the alternative hypothesis so that the expected number of edges under both distributions coincides. We prove that in both problems, when $q\gg n^{-1/2}$,
no test can perform better than random guessing, while for $q\ll n^{-1/2}$, there exist computationally efficient tests that guess correctly with high 
probability.
\end{abstract}

\input{plant}

\section{\bf Acknowledgements}

This material is based upon work supported in part by the National Science Foundation (NSF) under Grant No.~DMS-1928930, while the authors were in residence at the Simons Laufer Mathematical Sciences Institute (SLMath/MSRI) in Berkeley, California, participating in the Probability and Statistics of Discrete Structures program during Spring~2025.

\small 

\bibliographystyle{alpha}
\bibliography{eastereggs}

\appendix

\end{document}

%% file: preamble.tex
\usepackage{amsthm,amssymb,amsfonts,amsmath}
\usepackage[numbers, sort&compress]{natbib} 
\usepackage{subfigure, graphicx}
\usepackage[margin=2.75cm]{geometry}
\usepackage{setspace}
\usepackage{color}
\usepackage[normalem]{ulem}
\usepackage{mathtools}
\usepackage[dvipsnames]{xcolor}
\usepackage{hyperref}
\hypersetup{
    colorlinks=true, linkcolor=MidnightBlue,    
    citecolor=Sepia,           
    urlcolor=teal,             
}
\usepackage[nameinlink]{cleveref}

\providecommand{\R}{}

\providecommand{\N}{}

\renewcommand{\R}{\mathbb{R}}

\renewcommand{\N}{{\mathbb N}}

\newcommand{\e}{{\mathbf E}}

\newcommand{\va}{{\mathbf{Var}}}

\newcommand{\Ind}[1]{{\mathbf 1}_{[#1]}}

\newcommand{\dtv}{\mathrm{d}_{\mathrm{TV}}}

\newcommand{\bA}{\mathbf{A}} 
\newcommand{\bB}{\mathbf{B}} 
 
\newcommand{\bD}{\mathbf{D}}

\newcommand{\bG}{\mathbf{G}} 
\newcommand{\bH}{\mathbf{H}}

\newcommand{\bQ}{\mathbf{Q}}

\newcommand{\bT}{\mathbf{T}}

\newcommand{\bX}{\mathbf{X}} 
\newcommand{\bY}{\mathbf{Y}} 
\newcommand{\bZ}{\mathbf{Z}}

\newcommand{\bbG}{\mathbb{G}} 
\newcommand{\bbH}{\mathbb{H}}

\newcommand{\bbP}{\mathbb{P}} 
\newcommand{\bbQ}{\mathbb{Q}}

\providecommand{\ora}[1]{}
\renewcommand{\ora}[1]{\overrightarrow{#1}}
\DeclareRobustCommand{\SkipTocEntry}[5]{} 
\newcommand\defeq{\stackrel{\text{\tiny def}}{=}}

\reversemarginpar
\definecolor{clou}{rgb}{0.8,0.25,0.5125}

\crefname{ineq}{inequality}{inequalities}
\creflabelformat{ineq}{#2{\upshape(#1)}#3}

\newtheorem{thm}{Theorem}
\newtheorem{lem}[thm]{Lemma}

\newtheorem{fact}[thm]{Fact}
\newtheorem{claim}[thm]{Claim}

\numberwithin{equation}{section}
\numberwithin{thm}{section}

\graphicspath{ {./images/} }
\usepackage{wrapfig}

\newcommand{\Ber}{\mathrm{Ber}}
\newcommand{\Bin}{\mathrm{Bin}}
\newcommand{\Pois}{\mathrm{Pois}}

%% file: plant.tex
\section{Introduction}

The planted subgraph problem is a fundamental problem in combinatorial statistics that has received increasing attention, mostly because it is relatively simple, yet it possesses a wealth of intriguing and complex phenomena. Roughly speaking,  the planted subgraph problem asks whether one can detect or recover a hidden structured subgraph that has been embedded inside a large random graph.

More precisely, let $\bbG(n,p)$ denote the distribution of an Erd\H{o}s--R\'enyi random graph with $n$ vertices and edge probability $p$, and let $\bbH$ be a distribution over graphs on $n$ vertices with $e_H$ edges, representing the planted (or hidden) subgraphs. 
Consider sampling an $n$-vertex graph at random by first sampling a graph from $\bbG(n,p)$, then sampling $\bH \sim \bbH$, then adding a planted copy of $\bH$.
Let $\bbP$ ($\bbP$ for ``planted'') be the distribution on graphs induced by this sampling procedure.

In the \emph{detection problem}, one observes a graph that is either drawn from the planted distribution $\bbP$ or from the ``null'' distribution that is usually considered to be simply~$\bbG(n,p)$. One is asked to decide which of the two distributions the observed graph is generated by. 

In the early formulations of the problem (see Section \ref{sec:literature} below for references), $\bbH$ was chosen to be the uniform distribution over all cliques of size $k$ for some $k\ll n$.  The main, still unresolved, question is for what values of $k$ do there exist computationally efficient tests that identify the underlying distribution, with high probability, when $p=1/2$.

A closely related problem is that of \emph{recovery}, where one observes a graph drawn from the planted distribution $\bbP$, and the goal of the observer is to (partially or exactly) recover the planted subgraph $\bH$. 

In this paper we focus on the detection problem only. In particular, we consider planting either a perfect matching (assuming that $n$ is even) or a spanning tree. 
In other words, $\bbH$ is either the uniform distribution over all perfect matchings of $n$ vertices, or the uniform distribution over all  spanning trees. In the first case, $e_H=n/2$, while $e_H=n-1$ in the second. Since $e_H$ is linear in both cases, even when $p$ is bounded away from zero, one can detect the presence of a planted object simply by counting the edges of the observed graph. In order to avoid trivial detection by edge counting, we adjust the null hypothesis $\bbQ$ so that  $\bbP$ and $\bbQ$ have the same number of edges in expectation.
More precisely, we consider the hypothesis testing problem of distinguishing whether a graph is drawn from the planted distribution $\bbP$ defined above or the null distribution $\bbQ = \bbG(n,q)$, where $q \defeq p + (1-p)e_H/\binom{n}{2}$.

We consider the following standard formulation of the hypothesis testing problem. The statistician observes a graph $G$ on $n$ vertices. A test is a function $f:\{0,1\}^{\binom{n}{2}}\to \{0,1\}$, where $f(G)=0$ represents that the test guesses that $G$ is drawn from the null distribution $\bbQ$ and $f(G)=1$ means that the test guesses $\bbP$. The risk of a test $f$ is defined by the sum of the probabilities of the two types of error:
\[
   R(f) = \Pr_{\bG \sim \bbQ}[f(\bG)=1] + \Pr_{\bG \sim \bbP}[f(\bG)=0]~.
\]
It is well known that the optimal test minimizing the risk is the \emph{likelihood ratio test} defined by
\[
f^*(G) = 0 \ \text{ if and only if } \  L(G) < 1~,
\]
where
\[
L(G) 
\defeq \frac{d\bbP}{d\bbQ}(G)
\]
is the likelihood ratio. The risk of the optimal test is
\[
R(f^*) 
= 1- \frac{1}{2}
\e_{\bG \sim \bbQ} |L(\bG)-1|= 1-\dtv(\bbP,\bbQ)~,
\]
where $\dtv(\bbP,\bbQ)$ denotes the total variation distance. Thus, one can correctly detect the presence of a planted structure with high probability if and only if $\dtv(\bbP,\bbQ) \to 1$. 

The goal of the paper is to determine the value of $p$ (equivalently, of $q$) for which correct detection is possible, with high probability.

A standard tool for lower bounding $R(f^*)$ is the simple observation that, by the Cauchy-Schwarz inequality,
\[
R(f^*) \ge 1- \frac{1}{2}
\sqrt{\e_{\bG \sim \bbQ} (L(\bG)^2-1)}
= 1 -\frac{1}{2} \sqrt{\chi^2(\bbP,\bbQ)}~.
\]

Our main results establish that, in both problems considered, the high-probability detection threshold for $p$ is of the order of magnitude of $n^{-1/2}$.
More precisely, in both the planted perfect matching problem and the planted spanning tree problem, if $p= o(n^{-1/2})$, then $R(f^*) \to 0$, and if $p= \omega(n^{-1/2})$, then $R(f^*) \to 1$; 
see Theorems~\ref{thm:matching} and~\ref{thm:spanning_tree} below.

The order of magnitude of the threshold value is somewhat striking. 
On the one hand, a sparse and dispersed structure like a perfect matching (or a spanning tree) can be detected as soon as the random graph in which it is hidden has average degree $o(\sqrt{n})$. On the other hand, while without adjusting the edge probability of the null distribution, edge counting would make detection possible for any $p=o(1)$, after adjustment, no test can detect the planted structure when $p$ is of larger order than $n^{-1/2}$.
It is perhaps surprising that adjusting $p$ to hide the edge count has such a dramatic effect.

Our impossibility proofs rely on a careful bounding of the $\chi^2$ distance between $\bbP$ and~$\bbQ$. For the possibility results, we design simple tests that are able to tell the two distributions apart. Importantly, the tests are computationally efficient, and therefore there is no statistical/computational gap, unlike in many of the related planted subgraph problems.

\subsubsection*{Notation}
We use standard big-O notation, and we  write $o_n(1)$ to denote a function that goes to $0$ as $n \to \infty$.
Our convention is that random variables appear in boldface. 
We write $H \subset G$ to denote that $H$ is a subgraph of $G$.

\subsection{Related work}
\label{sec:literature}

The planted clique problem  \cite{MR1179827,MR1327775} has generated a massive amount of research, including on detection algorithms \cite{MR1662795,MR2735341,DeGuPe11,DeMo15}; 
statistical thresholds \cite{MR3664576,MR3670183,MR4751648,MR4441142}; and complexity lower bounds/reductions  \cite{brennan2018reducibility,brennan2020reducibility,MR4764820, MR4617369,MR2402475}.

There has been a substantial amount of work on recovery of planted perfect matchings in edge-weighted random graphs, 
where the weight distribution of the planted edges differs from that of the other edges \cite{MR4147946,MR4350971,MR4634336}; see also~\cite{pmlr-v291-gaudio25a}, which studies the recovery of a planted $k$-factor in a sparse Erd\H{o}s--R\'enyi random graph. The planting of subgraphs of varying size (rather than spanning subgraphs) has also been studied, notably in the very recent work~\cite{elimelech2025detecting}. 

The case when the planted subgraph is a tree was studied by Massouli{\'e}, Stephan, and Towsley \cite{MaStTo19}, who considered planting small trees in sparse Erd\H{o}s--Rényi graphs.

As far as we are aware, the only paper specifically focused on the planted spanning tree problem is  \cite{moharrami2025planted}, which studies the problem of recovering the planted spanning tree. 

The recent work \cite{MR4871715} studies all-or-nothing phenomena in planted graph detection in the Erd\H{o}s--R\'enyi setting, drawing connections to  Kahn--Kalai expectation thresholds.

The planted subgraph problem also has applications to secret-sharing schemes and other problems in cryptography \cite{MR4872490, MR4697144}. 

The distribution of counts of perfect matchings and spanning trees (and Hamiltonian cycles) in Erd\H{o}s-R\'enyi graphs $\bbG(n,p)$ were studied by Janson~\cite{janson1994numbers}. 
Janson showed that the counts follow a log-normal distribution so long as $p<1$ and $p = \Omega(n^{-1/2})$. 
To study the case $p \sim n^{-1/2}$, Janson studies a basis expansion of a reweighting of the matching count by a factor exponential in the number of edges.
Though he does not discuss it, this reweighting (as we will see in \Cref{sec:LR}) is equivalent to the likelihood ratio of the planted distribution to $\mathbb{G}(n,q')$ for some $q'$ (not necessarily corresponding to our edge-count-corrected null).

While we were finishing writing this manuscript, we learned of the very recent work by Wee and Mao~\cite{wee2025cluster}, who study planted matchings in a related but somewhat different setting.
In their work, they plant a matching whose size is random, with expectation bounded above by $\varepsilon n$ for some small $\varepsilon$. 
In this setting, they determine the detection threshold at $p = \Theta(n^{-1/2})$, and prove precise results at the threshold (in particular, log-normality of the likelihood ratio) using sophisticated cluster expansion methods (following \cite{janson1994numbers}). 
Their limitation on the expected size of the matching is required for the cluster expansion to converge. 
\subsection{Outline}
The rest of the paper is structured as follows. In Section \ref{sec:LR} we present properties of the likelihood ratio that turn out to be useful in the proofs of our impossibility results.
In Section \ref{sec:PM} we prove our result on the detection threshold for the planted perfect matching problem. In Section \ref{sec:ST} the analogous result for planted spanning trees is presented.

\section{Likelihood ratio for planted graphs}
\label{sec:LR}
Consider an experiment in which one samples $k$ graphs $\bH_1,\ldots,\bH_k$ independently from $\bbH$, and let $\mathrm{collisions}(\bH_1,\ldots,\bH_k)$ be the number of edges common to $\bH_1,\ldots,\bH_k$ (with multiplicity).
In this section we show that $k$th moments of the likelihood ratio $\frac{d\bbP}{d\bbQ}$ under $\bbQ$ can be understood by studying the moment generating function of $\mathrm{collisions}(\bH_1,\ldots,\bH_k)$.

For an $n$-vertex graph $G$ with $e_G$ edges, the likelihood ratio is given by the expression
\begin{align}
L(G) 
= \frac{d\bbP}{d\bbQ}(G) 
&= \frac{\Pr_{\bH \sim \bbH}[\bH \subset G] \cdot p^{e_G-e_H}(1-p)^{\binom{n}{2}-e_G}}{q^{e_G}(1-q)^{\binom{n}{2}-e_G}}\nonumber\\
&= \left(\frac{1-p}{1-q}\right)^{\binom{n}{2}}\left(\frac{\Pr_{\bH \sim \bbH}[\bH \subset G]}{\Pr_{\substack{\bH\sim\bbH\\ \bG \sim \bbG(n,p)}}[\bH \subset \bG]}\right) \left(\frac{p(1-q)}{q(1-p)}\right)^{e_G}.\label{eq:likelihood}
\end{align}

\newcommand{\deltanH}{\delta_{n}}
Writing $\deltanH = e_H / \binom{n}{2}$, simple algebraic manipulations allow us to write
\begin{align*}
    p &= \frac{q-\deltanH}{1-\deltanH}, &
    \frac{1-q}{1-p} &= 1-\deltanH, &
    \frac{p(1-q)}{q(1-p)} &=\frac{q-\deltanH}{q},
\end{align*}
and \Cref{eq:likelihood} above can be re-written as

\begin{align}
L(G) = \left(\frac{1}{1-\deltanH}\right)^{\binom{n}{2}}\left(\frac{\Pr_{\bH \sim \bbH}[\bH \subset G]}{\Pr_{\substack{\bH\sim\bbH\\ \bG \sim \bbG(n,p)}}[\bH \subset \bG]}\right) \left(1 - \frac{\deltanH}{q}\right)^{e_G}.\label{eq:likelihood-2}
\end{align}

\subsection{Likelihood ratio moments}

A change-of-measure allows us to set $k$th moments of the likelihood ratio $L(G)$ under $\bbQ$ proportional to moments of $\Pr_{\bH \sim \bbH}[\bH \subset \bG]$ for graphs $\bG$ from a product measure $\bbQ_k$.
In particular, continuing from the expression \cref{eq:likelihood-2}, defining $r_k \defeq (1-\deltanH/q)^k$,
\begin{align}
\e_{\bG \sim \bbQ} L(\bG)^k
&= \left(\frac{1}{1-\deltanH}\right)^{k\binom{n}{2}} \frac{\e_{\bG \sim \bbQ}\left[\Pr_{\bH \sim \bbH}[\bH \subset \bG]^k r_k^{e_{\bG}}\right]}{\Pr_{\substack{\bH \sim \bbH\\\bG \sim \bbG(n,p)}}[\bH \subset \bG]^k}\nonumber \\
&= \left(\frac{1}{1-\deltanH}\right)^{k\binom{n}{2}} \frac{Z_k \cdot \e_{\bG \sim \bbG(n,q_k)}\left[\Pr_{\bH \sim \bbH}[\bH\subset\bG]^k\right]}{\e_{\bG \sim \bbG(n,p)}\left[\Pr_{\bH \sim \bbH}[\bH \subset\bG]\right]^k}, \label{eq:homs}
\end{align}
with
\begin{align}
Z_k \defeq \left(1-q + q r_k\right)^{\binom{n}{2}}, \qquad \text{and} \qquad
q_k \defeq \frac{q r_k}{1-q + q r_k}.
\end{align}

We can further simplify the right-hand side of \cref{eq:homs}. 
For any $m \in \N$ and $a \in [0,1]$,
\[
\e_{\bG \sim \bbG(n,a)} \left[ \Pr_{\bH \sim \bbH}[\bH\subset\bG]^m \right]
= \e_{\bH_1,\ldots,\bH_m \sim \bbH} \left[ a^{|E(\bH_1 \cup \cdots \cup \bH_m)|} \right],
\]
where $\bH_1,\ldots,\bH_m \sim \bbH$ are mutually independent, and 
 $E(\bH_1 \cup \cdots \cup \bH_m)$ is the edge set of the \emph{simple} graph $\bigcup_{i=1}^m \bH_i$.
Applying this in the numerator and denominator of \cref{eq:homs},
\begin{align}
\frac{\e_{\bG \sim \bbG(n,q_k)}\left[\Pr_{\bH \sim \bbH}[\bH\subset\bG]^k\right]}{\e_{\bG \sim \bbG(n,p)}[\Pr_{\bH\sim \bbH}[\bH\subset\bG]]^k}
&= \left(\frac{q_k}{p}\right)^{k e_H} \e_{\bH_1,\ldots,\bH_k \sim \bbH} \left[q_k^{-\mathrm{collisions}(\bH_1,\ldots,\bH_k)}\right],\label{eq:coll}
\end{align}
where 
$\mathrm{collisions}(\bH_1,\ldots,\bH_k) 
\defeq k e_{H} - |E(\bH_1 \cup \cdots \cup \bH_k)|$ 
is the minimum number of arcs one must remove from the multiset $\bigcup_{i=1}^m E(\bH_i)$ in order to obtain a set.

Putting \cref{eq:homs,eq:coll} together,
\begin{align}
\e_{\bG \sim \bbQ} L(\bG)^k
&= 
\left(\frac{1}{1-\deltanH}\right)^{k\binom{n}{2}}\left(\frac{q_k}{p}\right)^{k e_H} Z_k \cdot  \e_{\bH_1,\ldots,\bH_k \sim \bbH} \left[q_k^{-\mathrm{collisions}(\bH_1,\ldots,\bH_k)}\right].\label{eq:coll-simple}
\end{align}

\section{Planting a perfect matching}
\label{sec:PM}

In this section we address the case when $n \in \N$ is even, and $\bbP$ is $\bbG(n,p)$ with a planted perfect matching, captured by choosing $\bbH$ to be the uniform distribution over perfect matchings of the complete graph $K_n$.
The main result of this section is the following.
\begin{thm}\label{thm:matching}
If $\bbP$ is $\bbG(n,p)$ with a planted perfect matching and $\bbQ$ is $\bbG(n,q)$ for $q = p + \frac{1-p}{n-1}$, then
\[
\lim_{n \to \infty} \dtv(\bbP,\bbQ) = \begin{cases}
0 & \text{when }p = \omega(n^{-1/2}),\\
1 & \text{when }p = o(n^{-1/2}).
\end{cases}
\]
\end{thm}

We first consider the case when $p = \omega(n^{-1/2})$ using the framework introduced in \Cref{sec:LR}, then give a simple hypothesis test which proves the case $p = o(n^{-1/2})$.
\begin{lem}
If $\bbH$ is the uniform distribution over perfect matchings of $K_n$ and $p = \omega(n^{-1/2})$ then $\lim_{n \to \infty} \dtv(\bbP,\bbQ) = 0$.
\end{lem}

\begin{proof}
It suffices to show that $\chi^2(\bbP,\bbQ) = \e_{\bG \sim \bbQ} (L(\bG)^2 - 1) =o_n(1)$.
We begin by instantiating \cref{eq:coll-simple}:
\begin{equation}\label{eq:coll-simple-matching}
\e_{\bG \sim \bbQ} L(\bG)^2 =
\left(\frac{1}{1-\deltanH}\right)^{2\binom{n}{2}}\left(\frac{q_2}{p}\right)^{n} \cdot Z_2 \cdot \e_{\bH_1,\bH_2 \sim \bbH} \left[q_2^{-\mathrm{collisions}(\bH_1,\bH_2)}\right].
\end{equation}
We next obtain an estimate of the moment generating function of $\mathrm{collisions}(\bH_1,\bH_2)$.
\begin{claim}\label{claim:pois-mgf}
    If $s n \gg \frac{\log n}{\log\log n}$, then
    \[
    \e_{\bH_1,\bH_2 \sim \bbH}\left[s^{-\mathrm{collisions}(\bH_1,\bH_2)}\right] \le \left(1 + o_n(1)\right) \e_{\bY \sim \mathrm{Pois}(\frac{1}{2})}\left[s^{-\bY}\right].
    \]
\end{claim}
\begin{proof}
    
We argue that the random variable $\bX = \mathrm{collisions}(\bH_1,\bH_2)$ is close to Poisson with mean $\frac{1}{2}$, such that the above is well-approximated by the Poisson moment generating function.
The set of perfect matchings of $K_n$ is in bijection with unordered partitions of $n$ into pairs, of which there are $(n-1)!!$.
By inclusion-exclusion, when $\bH_1,\bH_2$ are chosen independently and uniformly at random from this set, for any $k \in [n/2]$,
\begin{align*} 
\Pr [\bX = k] 
&= \frac{\binom{n/2}{k}\sum_{\ell=0}^{n/2} (-1)^\ell \cdot\binom{n/2 - k}{\ell} \cdot (n-2k - 2\ell-1)!!}{(n-1)!!}.
\end{align*}
We obtain two upper bounds on this probability.
First, for any $k \in [n/2]$ we obtain a crude bound from the $\ell=0$ summand:
\[
\Pr[\bX = k]
\le \binom{n/2}{k}\frac{(n-2k-1)!!}{(n-1)!!} 
= \frac{2^{n/2}}{k!\binom{n}{n/2}} \cdot \frac{1}{2^{n/2-k}} \binom{n-2k}{n/2-k} 
= O(\sqrt{n}) \cdot \frac{1}{2^k k!},
\]
where the final inequality follows from the fact that $\binom{2m}{m} = \frac{1}{\sqrt{\pi m }}2^{2m}(1+O(\tfrac{1}{m}))$.

Now, we obtain a finer bound for $k\ll n$.
If $k \ll n$, then letting $k^+$ be the smallest even integer at least as large as $k$, we over-estimate the probability by stopping the inclusion-exclusion at the $k^+$ term:
\begin{align*}
\Pr[\bX = k]
&\le \frac{\binom{n/2}{k}\sum_{\ell=0}^{k^+} (-1)^\ell \cdot\binom{n/2 - k}{\ell} \cdot (n-2k - 2\ell-1)!!}{(n-1)!!}\\
&= \frac{1}{k!} \frac{1}{\binom{n}{n/2}} \sum_{\ell=0}^{k^+} (-1)^\ell \frac{1}{\ell!} 2^{k+\ell} \binom{n-2k-2\ell}{n/2-k-\ell}\\
&= \frac{1}{k!} \frac{\sqrt{\pi n/2}}{2^{n}}\left(1+O(\tfrac{1}{n})\right) \cdot \sum_{\ell=0}^{k^+} (-1)^\ell \frac{1}{\ell!} \frac{2^{n-k-\ell}}{\sqrt{\pi(n/2-k-\ell)}}\left(1+O(\tfrac{1}{n})\right)\\
&= \frac{1}{2^k k!} \sum_{\ell=0}^{k^+} \left(-\frac{1}{2}\right)^\ell \frac{1}{\ell!}\left(1+O(\tfrac{k}{n})\right)\\
&\leq \left(1 + O(\tfrac{k}{n})\right) \cdot e^{-1/2} \frac{1}{2^k k!}.
\end{align*}

Let $\bY \sim \mathrm{Pois}(\frac{1}{2})$, with $\Pr[\bY = k] = e^{-1/2} \frac{1}{2^k k!}$.
Let $m \ll n$ be the smallest integer such that $s m > 100 \log n / \log\log n$.
We manipulate the moment generating function of $\bX$, applying the coarse bound when $k > m$ and the finer bound when $k \le m$ and obtain
\begin{align*}
    \sum_{k=0}^\infty \Pr[\bX = k] \left(\frac{1}{s}\right)^k
    &\le (1+o_n(1))\sum_{k=0}^{m} \Pr[\bY = k] \left(\frac{1}{s}\right)^k + \sum_{k=m+1}^n O(\sqrt{n})\frac{1}{2^k k!} \left(\frac{1}{s}\right)^k\\
    &\le (1+o_n(1))\sum_{k=0}^{m} \Pr[\bY = k] \left(\frac{1}{s}\right)^k + \sum_{k=m+1}^n  O(\sqrt{n k})\left(\frac{e}{k s}\right)^k\\
    &= (1+o_n(1)) \e\left[s^{-\bY}\right] + o_n(1)
    = (1+o_n(1))\e\left[s^{-\bY}\right]. \qedhere
\end{align*}
\end{proof}

In the case of a planted perfect matching we have that $\deltanH = \frac{1}{n-1}$, 
and a simple calculation shows that 
if $p = \omega(n^{-1/2})$, 
then $q_2 = \omega(n^{-1/2})$.  
Hence, \Cref{claim:pois-mgf} applies with $s = q_{2}$.
We next recall that when $\bY \sim \mathrm{Pois}(\lambda)$, then $\log(\e{s^{-\bY}}) = \lambda\left(\frac{1}{s} - 1\right)$.
Plugging this into \cref{eq:coll-simple-matching} and taking logarithms, we have that 
\begin{align}
    \log \e_{\bG \sim \bbQ} L(\bG)^2
    &\le o_n(1) - 2\binom{n}{2}\log\left(1-\deltanH\right)  + n\log\left(\frac{q_2}{p}\right)+ \log Z_2 + \frac{1}{2q_2} - \frac{1}{2}. \label{eq:log}
    \end{align}
    Some algebraic simplification gives us
\begin{align}
    q r_2 &= (q-\deltanH)(1-\deltanH/q),\label{eq:simplif}\\
    Z_2 &= (1-q+qr_2)^{\binom{n}{2}} = \left(1-2\deltanH + \tfrac{\deltanH^2}{q}\right)^{\binom{n}{2}},\nonumber\\
    q_2 &= Z_2^{-1/\binom{n}{2}} \cdot q r_2, \nonumber\\
    \frac{q_2}{p} &= (1-\deltanH)\left(1-\tfrac{\deltanH}{q}\right) Z_2^{-1/\binom{n}{2}}.\nonumber
\end{align}
So returning to \Cref{eq:log}, 
\begin{align*}
\log \e_{\bG\sim\bbQ} L(\bG)^2
&\le o_n(1) + \left(n - 2\binom{n}{2}\right) \log \left(1-\deltanH\right) + n \log\left(1-\tfrac{\deltanH}{q}\right)  \\
&\qquad \qquad + \left(1-\tfrac{n}{\binom{n}{2}}\right)\log Z_2 + \frac{1}{2 q_2} - \frac{1}{2}\\
&= o_n(1) - n(n-2)\log (1-\deltanH) + n \log(1-\tfrac{\deltanH}{q})  \\
&\qquad \qquad + \tfrac{n(n-3)}{2} \log \left(1 - 2\deltanH + \tfrac{\deltanH^2}{q}\right) + \frac{1}{2q_2} - \frac{1}{2}.
\end{align*}
Using the Taylor expansion of the logarithm $\log(1-x) = -x - \frac{x^{2}}{2} + o(x^{2})$ as $x \to 0$, 
and recalling that $\deltanH = \frac{1}{n-1}$ and $q = \omega(n^{-1/2})$, 
we have that 
$\log(1 - \deltanH) = - \deltanH - \deltanH^{2} / 2 + o( n^{-2})$, 
that 
$\log(1 - \deltanH / q) = - \deltanH / q + o(n^{-1})$, 
and that 
$\log(1 - 2\deltanH + \deltanH^{2} / q) 
= - 2 \deltanH + \deltanH^{2} / q - 2 \deltanH^{2} + o(n^{-2})$. 
Plugging these expressions into the display above, we see that the display above is equal to 
\[
o_n(1) 
- n(n-2)(-\deltanH - \tfrac{1}{2}\deltanH^2)
+ n \left(-\tfrac{\deltanH}{q}\right)
+ \tfrac{n(n-3)}{2}\left( -2\deltanH + \tfrac{\deltanH^2}{q} - 2\deltanH^2\right)
 + \frac{1}{2q_{2}} - \frac{1}{2}.
\]
This simplifies to 
\[
o_{n}(1) + \frac{1}{2q_{2}} - \left(1 + \Theta \left( \tfrac{1}{n} \right) \right) \frac{1}{2q}.
\]
Finally, using the fact that $q_{2} = q \left( 1 + \Theta\left( 1/(qn) \right) \right)$ 
and that $q = \omega(n^{-1/2})$, 
we obtain that the whole expression is $o_{n}(1)$. 
Hence $\chi^2(\bbP, \bbQ) = \e_{\bG \sim \bbQ} ( L(\bG)^2 - 1 ) = o_n(1)$, concluding the proof.
\end{proof}

\begin{lem}\label{lem:test-match}
If $\bbH$ is the uniform distribution over perfect matchings of $K_n$ and $p = o(n^{-1/2})$ then $\lim_{n \to \infty} \dtv(\bbP,\bbQ) = 1$.
\end{lem}
Before proving the lemma, we give some intuition from the (simpler) bipartite case.
Suppose we have a bipartite random graph on $2n$ vertices; then, the left-hand-side vertices' degrees are independent.
In the unplanted case, the law of the degree of each left vertex is close to $\bA \sim \Pois(\lambda)$ for $\lambda = nq$ when $q,p \ll n^{-1/2}$.
In the planted case, the law is instead close to $\bB \sim 1+\Pois(\lambda - 1)$.
The means of $\bA$ and $\bB$ are the same, but the variances are $\lambda$ and $\lambda-1$ respectively.
Further, the standard deviations of $\bA' = (\bA - \e[\bA])^2$ and $\bB' = (\bB - \e[\bB])^2$ are also $\Theta(\lambda)$.
Hence the $n = \Omega(\lambda^{2})$ i.i.d.\ samples furnished by the left vertex degrees suffice to distinguish the two distributions. 
In the non-bipartite case the degrees are not independent, but the idea of the proof is effectively the same.

\begin{proof}[Proof of \Cref{lem:test-match}]
We design a hypothesis test that distinguishes $\bbP$ and $\bbQ$ with high probability when $p = o(n^{-1/2})$ (and note that $p = o(n^{-1/2})$ implies $q = o(n^{-1/2})$).
Choose a near-regular tournament of $K_n$, that is, a partition $P_1,\ldots,P_n$ of the edges of $K_n$ such that all edges in $P_i$ are incident on vertex $i$, and each part contains either $\lceil (n-1)/2 \rceil$ or $\lfloor (n-1)/2 \rfloor$ edges.

For each $i \in [n]$, let $\bX_i = \sum_{e \in P_i} \Ind{e \in \bG}$, and let $\bY_i = (\bX_i - \e[\bX_i])^2$ (note that $\e\bX_i$ is the same under $\bbP$ and $\bbQ$).
The hypothesis test evaluates the statistic $\bY = \frac{1}{n}\sum_{i=1}^n \bY_i$, selecting $\bbQ$ if and only if the value is larger than $\e_\bbQ\bY - \frac{1}{8}$.
We  argue that the difference of means under $\bbP$ and $\bbQ$ is at least $\frac{1}{4}-o_n(1)$, and that $|\bY-\e\bY| = o_n(1)$ with high probability under both $\bbP$ and $\bbQ$, from which the claim follows.

Under $\bbQ$, the $\bY_i$ are independent, and each $\bX_i$ has the law of $\Bin(|P_i|,q)$.
Hence the mean of each term is $\e_{\bbQ} \bY_i = |P_i|q(1-q)$, and the mean of our statistic is $\e_{\bbQ}[\bY] = \frac{n-1}{2}q(1-q)$.
We further have that the statistic concentrates around its mean, because it is a sum of independent random variables with bounded variance: for each $i \in [n]$, $\sqrt{\e_{\bbQ} (\bY_i - \e[\bY_i])^2} = O(nq + 1)$.
Thus, applying standard concentration inequalities, with high probability, $|\bY - \e_\bbQ \bY| \le \frac{O(nq+1)}{\sqrt{n}} =o_n(1)$ when $q,p \ll n^{-1/2}$.

Under $\bbP$, conditioned on a specific planted matching $H$, the $\bY_i$ are independent, with the $n/2$ of the $\bX_i$ for which $P_i \cap E(H) = \emptyset$ following the law of $\Bin(|P_i|,p)$ and the rest following the law of $1+\Bin(|P_i|-1,p)$. 
A calculation yields that
\begin{align*}
\e_{\bbP} \left[ \bY_i \mid H\right] 
&= \e_{\bbP} \left(\sum_{e \in P_i \setminus H} (\Ind{e\in \bG} - p)\right)^2 + \left( \sum_{e \in P_i \cap E(H)} (\Ind{e\in \bG}-p) - |P_i|(q-p)\right)^2\\
&= (|P_i\setminus E(H)|)p(1-p) + (|P_i \cap E(H)|(1-p) -(q-p)|P_i|)^2 \\
&= (|P_i| - \Ind{|E(H)\cap P_i| = 1})p(1-p) + \tfrac{1}{4}+o_n(1),
\end{align*}
where we used that $|P_{i}|(q-p) = 1/2 + o_{n}(1)$, 
and so 
\[
\e_{\bbP}\bY = \e_{\bbH}\e_{\bbP}[\bY \mid \bH] = \frac{n-1}{2}p(1-p) - \tfrac{p(1-p)}{2} + \tfrac{1}{4}(1+o_n(1)).
\]
Concentration again follows from the fact that $\bY$ is a sum of independent random variables conditioned on $H$, each with bounded variance, $\sqrt{\e_{\bbP} (\bY_i - \e[\bY_i])^2} = O(np + 1)$, so with high probability $|\bY - \e_{\bbP}\bY| \le \frac{O(np + 1)}{\sqrt{n}} = o_n(1)$.

Finally, we claim that the means are $\frac{1}{4} - o_n(1)$ separated.
Recalling again that $\delta_n = \frac{1}{n-1}$ and $p = \frac{q-\deltanH}{1-\deltanH}$, $1-p = \frac{1-q}{1-\deltanH}$,
\begin{align*}
\e_{\bbP} \bY - \e_{\bbQ} \bY 
&= \frac{(1-p)^2}{4}(1+o_n(1)) - \frac{p(1-p)}{2} + \frac{n-1}{2}\left(p(1-p) - q(1-q)\right)\\
&= o_n(1) + \frac{1}{4} + \frac{n-1}{2}\left(\frac{q-\delta_n}{(1-\delta_n)}\frac{1-q}{1-\delta_n} - q(1-q)\right)\\
&= o_n(1) + \frac{1}{4} + \frac{n-1}{2}\left(-\delta_n + o(n^{-1}) \right)\\
&\le - \frac{1}{4} + o_n(1). \qedhere
\end{align*}
\end{proof}

\section{Planting a Spanning Tree}
\label{sec:ST}

Suppose now that in a sample from $\bbP$, the planted graph $\bH$ is a uniformly random spanning tree of $K_n$. In this section we prove the following result. 

\begin{thm}\label{thm:spanning_tree}
If $\bbP$ is $\bbG(n,p)$ with a planted spanning tree and $\bbQ$ is $\bbG(n,q)$ for $q = p + \frac{2(1-p)}{n}$, then
\[
\lim_{n \to \infty} \dtv(\bbP,\bbQ) = \begin{cases}
0 & \text{when }p = \omega(n^{-1/2}),\\
1 & \text{when }p = o(n^{-1/2}).
\end{cases}
\]
\end{thm}

To upper bound the total variation distance when $p = \omega(n^{-1/2})$, again we use the framework set up in \Cref{sec:LR}.
To control the moment generating function of the collisions, we can make use of the fact that the edges of a uniformly random spanning tree are negatively associated.

\begin{fact}\label{fact:neg-assoc}
    If $J = (E,V)$ is a graph and $\bH$ is a uniformly random spanning tree of $J$, then the random variables $\{\Ind{e \in \bH}\}_{e \in E}$ are negatively associated.
\end{fact}

For a subset $E' \subset E$, let the random variable $\bX_{E'} = \sum_{e \in E'} \Ind{e \in \bH}$, and let $\bY_{E'} = \sum_{e \in E'} \Ber(\Pr[e \in \bH])$ be a sum of independent Bernoulli random variables which are marginally identical to the $\Ind{e \in \bH}$.
Since the edges of the spanning tree are negatively associated, it holds \cite{Shao00} that for any $t \in \R$,
\[
\e[\exp(t \bX_{E'})] \le \e[\exp(t \bY_{E'})].
\]

\begin{lem}
If $\bbH$ is the uniform distribution over spanning trees of $K_n$ and $p = \omega(n^{-1/2})$ then $\lim_{n \to \infty} \dtv(\bbP,\bbQ) = 0$.
\end{lem}

\begin{proof}
Let $\bH_1,\bH_2$ be two spanning trees of $K_n$ chosen uniformly at random. 
We can write 
\[
\mathrm{collisions(\bH_1,\bH_2)} = \sum_{e \in \bH_1} \Ind{e \in \bH_2}.
\]
From \Cref{fact:neg-assoc},
\[
\e q_2^{-\mathrm{collisions}(\bH_1,\bH_2)} \le \e q_2^{-\bY} = \left(1 + \frac{2}{n}\left(\frac{1}{q_2}-1\right)\right)^{n-1}
\]
for $\bY \sim \Bin(n-1,\frac{n-1}{\binom{n}{2}})$.
Plugging this into \Cref{eq:coll-simple} with $k = 2$ and using $e_{H} = n-1$,  
\begin{align*}
\log \e_{\bG \sim \bbQ}L(\bG)^2 
&\leq -n(n-1)\log\left(1-\delta_n\right) + 2(n-1) \log \left(\frac{q_2}{p}\right) + \log Z_2 \\
&\qquad \qquad + (n-1)\log\left(1+\frac{2}{n}\left(\frac{1}{q_2}-1\right)\right).
\intertext{Applying again the simplifications from \Cref{eq:simplif}, this is further equal to}
&= (2(n-1)-n(n-1))\log\left(1-\delta_n\right) + 2(n-1)\log\left(1-\tfrac{\deltanH}{q}\right) \\&\qquad\qquad + \left(\binom{n}{2} - 2(n-1)\right)\log \left(1-2\deltanH + \tfrac{\deltanH^2}{q}\right) \\
&\qquad \qquad + (n-1)\log\left(1+\frac{2}{n}\left(\frac{1}{q_2}-1\right)\right). 
\intertext{Taking a Taylor expansion of the logarithm, and using that $\deltanH = \frac{2}{n}$, that $p,q=\omega(n^{-1/2})$, and that $q_{2} = q \left( 1 + \Theta(1/(qn)) \right)$, the expression above further simplifies to}
&= o_n(1) -(n-2)(n-1)\left(-\delta_n - \tfrac{1}{2}\deltanH^2\right) + 2(n-1)\left(-\tfrac{\deltanH}{q}\right) \\&\qquad\qquad + \frac{(n-4)(n-1)}{2}\left(-2\deltanH + \tfrac{\deltanH^2}{q} -2\deltanH^2\right) \\
&\qquad \qquad + (n-1)\left(- \tfrac{2}{n} + \tfrac{2}{nq}\right)\\
&= o_n(1).
\end{align*}
Hence $\chi^2(\bbP,\bbQ) = \e_{\bG \sim \bbQ}L(\bG)^{2} - 1 = o_n(1)$, and the conclusion follows.
\end{proof}

To lower bound the total variation distance in the case $p = o(n^{-1/2})$, we use a hypothesis test almost identical to the one used for matchings.

\begin{lem}
If $\bbH$ is the uniform distribution over spanning trees of $K_n$ and $p = o(n^{-1/2})$ then $\lim_{n \to \infty} \dtv(\bbP,\bbQ) = 1$.
\end{lem}
\begin{proof}
Just as in the proof of \Cref{lem:test-match}, let $P_1,\ldots,P_n$ be a partition of edges of $K_n$ corresponding to a regular or near-regular tournament.
Let $\bX_i = \sum_{e \in P_i} \Ind{e \in \bG}$, let $\bY_i = (\bX_i - \e\bX_i)^2$, and let $\bY = \frac{1}{n}\sum_{i=1}^n \bY_i$.
Just as before, our hypothesis test is to compute $\bY$ and output $\bbP$ if and only if $\bY < \e_{\bbQ}\bY - \frac{1}{8}$.

Exactly as in the proof of \Cref{lem:test-match}, $\e_\bbQ[\bY] = \frac{n-1}{2}q(1-q)$, and with high probability $|\bY - \e_{\bQ}\bY| = o_n(1)$ when $q,p\ll n^{-1/2}$.

Under $\bbP$, for convenience, we write $\bY_i = (\bZ_i + \bT_i)^2$, with $\bZ_i = \sum_{e \in P_i \setminus E(\bH)} (\Ind{e \in \bG} - p)$ and $\bT_i = |P_i \cap E(\bH)|(1-p) + |P_i|(p-q)$. 
Conditioned on $\bH = H$, the $\bT_i$ are fixed constants, and each $\bZ_i$ is independent and mean zero.
It follows that
\begin{align*}
\e_{\bbP} \left[ \bY_i \mid H\right] 
&= \e_{\bbP}[\bZ_i^2\mid H] + T_i^2\\
&= |P_i\setminus E(H)|p(1-p) + (|P_i \cap E(H)|(1-p)- |P_i|(q-p))^2\\
&= |P_i \setminus E(H)|p(1-p) + (|P_i \cap E(H)|(1-p) - 1 + o_n(1))^2,\\
\intertext{where we used that $|P_i|(q-p) = 1 + o_{n}(1)$, and that}
\va_{\bbP}\left[\bY_i \mid H\right]
&\le 2\va_{\bbP} [\bZ_i^2 \mid H] + 2 \va_{\bbP} [2\bT_i\bZ_i \mid H]\\
&= O(np+1)^2 + O(np + 1)\cdot (|P_i \cap E(H)|(1-p) - 1 + o_n(1))^2,
\end{align*}
and finally that 
$\mathrm{Cov}_\bbP(\bY_i,\bY_j \mid H) = 0$ for $i \neq j$. 
From this, it follows that conditioned on~$H$,
\begin{align*}
\e_{\bbP}[\bY \mid H]
&= \frac{(n-2)(n-1)}{2n}p(1-p) + \e_{i \sim \mathrm{Unif}([n])}(|P_i \cap E(H)|(1-p) - 1 + o_n(1))^2\\
\va_{\bbP}(\bY \mid H),
&= \frac{O(np+1)^2}{n} +\frac{O(np+1)}{n}\e_{i \sim \mathrm{Unif}([n])}(|P_i \cap E(H)|(1-p) - 1 + o_n(1))^2).
\end{align*}
We now invoke properties of the degree distribution of $\bbH$. 
Using the Pr\"{u}fer code method for sampling from $\bbH$, the degrees $\bD_i,\bD_j$ of vertices $i,j$ in $H$ are jointly distributed as shifted multinomial, so that $(\bD_i -1,\bD_j - 1)$ are sampled by performing $n-2$ trials with probability $1/n$ of increasing $\bD_i-1,\bD_j-1$ respectively.
Hence any $O(1)$ joint moments of $\bD_i,\bD_j$ are, up to $(1\pm o_n(1))$ factors, equal to the moments of independent $\Pois(1)$ random variables.
Let $\bB_i = |P_i \cap E(\bH)|$; by symmetry, each edge of $E(H)$ incident on $i$ is equally likely to be included in $P_i$ as its other endpoint, and any $O(1)$ joint moments of $\bB_i,\bB_j$ are equal to the joint moments of independent $\Pois(\frac{1}{2})$ random variables up to $(1\pm o_n(1))$ factors.
Hence 
\[
\e_{\bbH} (|P_i \cap E(\bH)|(1-p) - 1 + o_n(1))^2 = \frac{3}{4} + o_n(1),
\]
and using our formula for the conditional expectation above,
\[
\e_{\bbP}\bY = \e_{\bbH}\e_{\bbP}[\bY \mid \bH] = \frac{n}{2}p(1-p) + \tfrac{3}{4} +o_n(1).
\]
We bound the variance of $\bY$ using the law of total variance:
\begin{align*}
\va_{\bbP}(\bY) &=
\e_{\bbH} \va_{\bbP}(\bY\mid\bH) + \va_{\bbH} \e_{\bbP}[\bY \mid \bH].
\end{align*}
Above, we have computed the conditional variance, and using that $\e[\bB_i^2] = O(1)$, when $p \ll n^{-1/2}$, 
we have that $\e_\bbH \va_{\bbP}(\bY \mid \bH) = o_n(1)$.
Finally, we upper bound the variance of the conditonal expectation. 
Using that the $\bB_i,\bB_j$ are asymptotically independent Poisson,
\begin{align*}
\va_{\bbH} \e_{\bbP}[\bY \mid \bH]
&= \e_{i,j\sim [n]} \e_{\bbH} ((\bB_i(1-p)-1)^2 - \frac{3}{4} + o_n(1)) ((\bB_j(1-p)-1)^2 - \frac{3}{4} + o_n(1))\\
&= o_n(1).
\end{align*}
So from Chebyshev's inequality, with high probability under $\bbP$, $|\bY - \e[\bY]| = o_n(1)$.

Finally, we compare the mean of $\bY$ under $\bbP$ and $\bbQ$, and we have that
\begin{align*}
\e_\bbP \bY - \e_{\bbQ}\bY
&= \frac{n}{2}(p(1-p) - q(1-q)) + \frac{3}{4} + o_n(1)\\
&= \frac{n}{2} \left(\frac{q-\delta_n}{1-\delta_n}\left(1-\frac{q-\delta_n}{1-\delta_n}\right) - q(1-q)\right) + \frac{3}{4} + o_n(1)\\
&= \frac{n}{2}\left(- \frac{2}{n} + o(\frac{1}{n})\right) + \frac{3}{4} + o_n(1)
= - \frac{1}{4} +o_n(1).
\end{align*}
Hence the hypothesis test succeeds with high probability under both $\bbP$ and $\bbQ$, from which the statement follows.
\end{proof}